\documentclass[12pt]{article}
\usepackage{latexsym}
\usepackage{amssymb,amsmath}
\usepackage{graphicx}
\RequirePackage{tikz}
\usepackage{epsfig}
\usepackage{amsfonts}
\newtheorem{theorem}{Theorem}[section]

\newtheorem{corollary}[theorem]{Corollary}
\newtheorem{proposition}[theorem]{Proposition}

\newcommand{\qed}{\hfill $\square$ \bigskip}

\begin{document}

\title{Global forcing number for maximal matchings in corona products}

\author{ Sandi Klav\v{z}ar$^{a,b,c}$ \and Mostafa Tavakoli$^{d}$\footnote{Corresponding author} \and Gholamreza Abrishami$^{d}$} 

\date{}

\maketitle
\vspace{-0.8 cm}
\begin{center}
	$^a$ Faculty of Mathematics and Physics, University of Ljubljana, Slovenia\\
	{\tt sandi.klavzar@fmf.uni-lj.si}\\
	\medskip

	$^b$ Faculty of Natural Sciences and Mathematics, University of Maribor, Slovenia\\
	\medskip

	$^c$ Institute of Mathematics, Physics and Mechanics, Ljubljana, Slovenia\\
	\medskip

$^d$ Department of Applied Mathematics, Faculty of Mathematical Sciences,\\
Ferdowsi University of Mashhad, P.O.\ Box 1159, Mashhad 91775, Iran\\
{\tt m$\_$tavakoli@um.ac.ir\\
 \tt gh.abrishamimoghadam@mail.um.ac.ir
\medskip
}

\end{center}

\begin{abstract}
A global forcing set for maximal matchings of a graph $G=(V(G), E(G))$ is a set $S \subseteq E(G)$ such that  $M_1\cap S \neq M_2 \cap S$ for each pair of maximal matchings $M_1$ and $M_2$ of $G$.  The smallest such set is called a minimum global forcing set, its size being the global forcing number for maximal matchings $\phi_{gm}(G)$ of $G$. In this paper, we establish lower and upper bounds on the forcing number for maximal matchings of the corona product of graphs. We also introduce an integer linear programming model for computing the forcing number for maximal matchings of graphs.
\end{abstract}


\medskip\noindent
{\bf AMS Subj.\ Class:} 05C70; 90C10

%%%%%%%%%%%%%%%%%%%%%%%%%
\section{Introduction}
\label{sec:intro}
%%%%%%%%%%%%%%%%%%%%%%%%%

Matchings represent one of the most important concepts of graph theory with many applications~\cite{lovasz}. In particular, (perfect) matchings are extremely important in theoretical chemistry, where perfect matchings bear the alternative name Kekul\'{e} structures, cf.~\cite{cigher-2009, feng-2021, zhang-2018, zhao-2020}. 

The investigation of forcing sets in the theory of graph matchings also has a considerable history, largely born of the study of resonant structures in chemical graphs. The list of articles~\cite{adams-2004, diwan-2019, sedlar-2012, tavakoli} is just a small selection of research in this field. In this paper we are interested in global forcing sets for maximal matchings, the concept introduced in~\cite{global} as an extension of the idea of global forcing sets (and global forcing numbers) for perfect matchings. We proceed as follows. In the rest of the introduction we first introduce the required terminology and notation from the matching theory, and then present the corona product and present some related notation to be used later on. In the main section of the paper we prove several upper and lower bounds on the global forcing numbers for perfect matchings of corona products. As a side result we also derive a formula for the matching number of corona products. In the concluding section we introduce an integer linear programming model for computing the forcing number for maximal matchings of graphs. 

\subsection{Matching terminology}

Let $G=(V(G), E(G))$ be a graph. A \emph{matching} in $G$ is a set of pairwise non-adjacent edges. A \emph{maximal matching} is one which cannot be extended to a larger matching. Each vertex incident with an edge of a matching $M$ of $G$ is said to be {\em saturated} by $M$. A \emph{maximum matching} is a maximal matching which saturates as many vertices as possible.  The size of a maximum matching in $G$ is called the {\em matching number} of $G$ and denoted by $\nu(G)$.  A \emph{perfect matching} of $G$ is a maximum matching which saturates every vertex of $G$. The smallest size of a maximal matching in $G$ is called the \emph{saturation number} of $G$ and denoted by $s(G)$. See~\cite{Yannakakis} for an application of smallest maximal matchings related to a telephone switching network and~\cite{Biedl, lovasz, tratnik, zito} for relations between the saturation number and the matching number.

A \emph{global forcing set for maximal matchings} of $G$ is a set $S \subseteq E(G)$ such that  $M_1\cap S\neq M_2 \cap S$ for any two maximal matchings $M_1$ and $M_2$ of $G$. As we are dealing only with matchings, in the rest of the paper we will shorten this long naming by saying that $S$ is a {global forcing set} of $G$. A smallest global forcing is called a \emph{minimum global forcing set} and the size of it, denoted by $\phi_{gm}(G)$, is called the \emph{global forcing number}~\cite{tavakoli, global}. 

\subsection{Corona products}

If $G$ is a graph, then its order will be denoted by $n(G)$. We will also use the convention that if $n$ is a positive integer, then $[n] = \{1,\ldots, n\}$. 

Let $G$ and $H$ be graphs with $V(G) = \{g_1, \ldots,g_{n(G)}\}$ and $V(H) = \{h_1,\ldots, h_{n(H)}\}$.  The \emph{corona product} of $G$ and $H$, denoted by $G\circ H$, is a graph obtained from the disjoint union of a copy of $G$ and $n(G)$ copies of $H$, denoted by $H_i$, $i\in [n(G)]$. The product $G\circ H$ is then obtained by making adjacent $g_i$ to every vertex in $H_i$ for each $i\in [n(G)]$, cf.~\cite{Klavzar-corona, kristiana-2020, yoong-2021}. 

For $i\in [n(G)]$, let $V(H_i)=\{h_1^i, \ldots, h_{n(H)}^i\}$. Then the corona product $G\circ H$ can be formally defined as follows. It vertex set is 
$$V(G \circ H) = V(G) \cup \Bigg [ \bigcup\limits_{i=1}^{n(G)} V(H_i) \Bigg ]\,,$$
and its edge set is $E(G \circ H) = E_G \cup E_H \cup E_{H,G}$, where
\begin{align*}
E_G&=E(G),\\\\
E_H&= \Bigg [ \bigcup\limits_{i=1}^{n(G)} E(H_i) \Bigg ], \\\\
E_{H,G}&= \Bigg [ \bigcup\limits_{i=1}^{n(G)}\  \bigg \{g_ih_j^i:\ j\in [n(H)] \bigg\} \Bigg ].
\end{align*}
We use the notations $E_G$, $E_H$, and $E_{H,G}$ for the edges of the corona product of two given graphs $G$ and $H$  throughout the paper. 

%%%%%%%%%%%%%%%%%%%%%%%%%
\section{Bounds on the global forcing number of corona products }
\label{sec:main}
%%%%%%%%%%%%%%%%%%%%%%%%%

In this section we give upper and lower bounds on the global forcing number of corona products. For this sake we first determine the matching number of corona products, a result of independent interest. 

\begin{proposition}
\label{prop:nu}
Let $G$ and $H$ be two graphs. If $H$ has a perfect matching, then 
\[\nu(G\circ H)= \nu(G) + n(G)\nu(H) \,,\]
otherwise,
\[\nu(G \circ H)= n(G) + n(G)\nu(H)\,.\]
\end{proposition}

\begin{proof}
First suppose that $H$ has a perfect matching. Let $M$ be a matching of $G \circ H$ with $|M|=\nu(G \circ H)$. Suppose that $M \cap E_{G,H} \ne \emptyset$ and let $ e=g_ih_j^i \in (M \cap E_{G,H})$. Since $H$ has a perfect matching, $|M \cap E(H_i)| < \nu(H)$. Let  $M^*$ be a perfect matching of $H_i$ and set 
\[ M' = \bigg [ M \setminus \big [\{e\}\cup \big (M \cap E(H_i) \big) \big] \bigg] \cup M^* \,. \]
Then $M'$ is a matching with $|M'| \ge |M|$, and therefore, since $M$ is maximum, $|M'|  = |M| = \nu(G\circ H)$. Repeating this process as many times as necessary, we arrive at a maximum matching $M''$ of $G\circ H$ with $M'' \cap E_{G,H}=\emptyset$. But now it is clear that $|M'' \cap E(H_i)|  = \nu(H)$ for each $i\in [n(G)]$ and that  $|M'' \cap E_G|  = \nu(G)$. We conclude that  $\mu(G\circ H) = |M''| = \nu(G) + n(G)\nu(H)$. 

Suppose second that $H$ does not admit a perfect matching. A matching of $G\circ H$ in each $H_j$, $j\in [n(G)]$, saturates as most $2\nu(H) + 1$ vertices, while a matching of $G\circ H$ clearly saturates at most $n(G)$ vertices from $G$. It is straightforward to construct a matching that saturates that many vertices: in each $H_j$ take a maximum matching of $H$, and then for each $j\in [n(G)]$ add an edge between an unsaturated vertex of $H_j$ and $g_j$. Such a matching contains $n(G) + n(G)\nu(H)$ edges and we are done. 
\qed
\end{proof}

With Proposition~\ref{prop:nu} in hand, we can bound the global forcing number of corona products from the above as follows. 

\begin{theorem}
\label{thm:upper}
Let $G$ and $H$ be two graphs. If $H$ has a perfect matching, then 
\[\phi_{gm}(G\circ H) \le|E(G \circ H)| - \nu(G) - n(G)\nu(H)\,,\]
otherwise,
\[\phi_{gm}(G\circ H) \le |E(G \circ H)|- n(G) - n(G)\nu(H)\,.\]
\end{theorem}

\begin{proof}
In~\cite[Corollary 5]{global} Vuki\v{c}evi\'c et al.\ observed that the complement of a matching is a global forcing set and consequently, $\phi_{gm}(X) \le m(X) - \nu(X)$ holds for a graph $X$, cf.~\cite[Theorem 7]{global}. Combining this bound with Proposition~\ref{prop:nu}, the result follows. 
\qed
\end{proof}

As a small example consider the graph $Y$ from Fig.~\ref{fig:graphY}. Consider the subgraph $K_2$ of $Y$ induced by the edge $e$ to see that $Y = K_2\circ K_2$. If $e$ belongs to a maximal matching of $Y$, then this matching is unique. On the other hand, there are exactly $9$ maximal matchings of $Y$ that do not contain $e$. Denoting by $\Psi(G)$ the total number of maximal matchings of a graph $G$, we thus have $\Psi(Y) = 10$. From~\cite[Proposition 2]{global} we know that $\phi_{gm}(G) \ge \lceil \log_2 \Psi(G)\rceil$ holds for a graph $G$, hence  $\phi_{gm}(Y) \ge 4$. On the other hand, the first inequality of Theorem~\ref{thm:upper} gives  $\phi_{gm}(Y) \le 4$.  

\begin{figure}[ht!]
\begin{center}
\begin{tikzpicture}[scale=0.7,style=thick,x=1cm,y=1cm]
\def\vr{3pt}

% define vertices
\coordinate(x1) at (0,-1);
\coordinate(x2) at (0,1);
\coordinate(x3) at (2,0);
\coordinate(x4) at (5,0);
\coordinate(x5) at (7,-1);
\coordinate(x6) at (7,1);
%edges
\draw (x1) -- (x2) -- (x3) -- (x4) -- (x5) -- (x6) -- (x4);
\draw (x1) -- (x3);
%  vertices
\foreach \i in {1,2,3,4,5,6}
{
\draw(x\i)[fill=white] circle(\vr);
}
% labels
\draw(x3)++(1.5,-0.1) node[below] {$e$};
\end{tikzpicture}
\end{center}
\caption{Graph $Y$}
\label{fig:graphY}
\end{figure}

With respect to the above example we mention that a recurrence relation for the sequence $p_n = \Psi(P_n\circ P_1)$ was derived in~\cite[Proposition 7.3]{doslic-2016}, and that in~\cite[Proposition 7.4]{doslic-2016} it was proved that $\Psi(P_n\circ K_3) = 3^{n+1}F_{n+2}$, where $F_n$ is the $n$th Fibonacci number.  

For the second inequality of Theorem~\ref{thm:upper} consider the corona product $K_2\circ K_1$ which is isomorphic to $P_4$. Clearly $\phi_{gm}(P_4) = 1$ which is also the value given by the second inequality. It would be interesting to find more sharpness examples for both bounds of Theorem~\ref{thm:upper}.  

We can also bound the global forcing number of $G\circ H$ in terms of $\phi_{gm}(G)$ and $\phi_{gm}(H)$. 

\begin{theorem}
If $G$ and $H$ are graphs, then 
\[\phi_{gm}(G\circ H) \le \phi_{gm}(G) + n(G)\phi_{gm}(H) + n(G)n(H).\]
\end{theorem}

\begin{proof}
Let $S_G$ be a minimum global forcing set of $G$ and let $S_{H_i}$, $i\in [n(G)]$, be a minimum global forcing set of $H_i$.  Set
\[S=S_G\cup \left(\bigcup^{n}_{i=1}S_{H_i} \right) \cup E_{G,H}.\]
Clearly, $|S|=\phi_{gm}(G) + n\phi_{gm}(H)+n(G)n(H)$. We claim that $S$ is a global forcing set in $G \circ H$. For this sake assume that there are maximal matchings $M_1$ and $M_2$ of $G \circ H$  such that $M_1\neq M_2$ and $M_1\cap S=M_2\cap S$. Since each of $M_1\cap E(H_i)$ and $M_2 \cap E(H_i)$ is a maximal matchings of $H_i$, $i\in [n(G)]$, and as $S\cap E(H_i) = S_{H_i}$ is a minimum global forcing of $E_i$, we have $M_1\cap E(H_i) = M_2 \cap E(E_i)$  for each $i\in [n(G)]$. Similarly we get that $M_1\cap E(G) = M_2 \cap E(G)$. Since $M_1\neq M_2$, we have $M_1\cap E_{G,H} \ne M_2\cap E_{G,H}$. But $E_{G,H} \subseteq S$ we then get that $M_1\cap S \ne M_2\cap S$. 
\qed
\end{proof}

A graph $G$ is \emph{randomly matchable} if every matching extends to a perfect matching, equivalently, if every maximal matching is a perfect matching. For corona products in which the second factor is randomly matchable, we have the following lower bound. 

\begin{theorem}
\label{thm:randomly}
If $G$ is a graph and $H$ a randomly matchable graph, then 
\[\phi_{gm}(G\circ H) \ge \phi_{gm}(G) + n(G)\phi_{gm}(H)+\frac{n(G)n(H)}{2}\,.\]
\end{theorem}

\begin{proof}
In what follows, $M_{H_i}$ will denote a maximal matching of $H_i$, $i\in [n(G)]$, and $M_G$ a maximal matching of $G$. As $H$ is randomly matchable, each $M_{H_i}$ is a perfect matching of $H_i$. 

Let $S$ be a minimum global forcing set of $G\circ H$. Thus, $|S|=\phi_{gm}(G\circ H)$. We claim that for each $i\in [n(G)]$, we have
\[|S \cap E(H_i)| \ge \phi_{gm}(H). \]
Suppose to the contrary that there exists an integer $j\in [n(G)]$ such that $|S \cap E(H_j)| < \phi_{gm}(H)$. Then there exist two maximal (equivalently, perfect) matchings $M_1^+$ and $M_2^+$ of $H_j$ such that
$[M_1^+ \cap \big(S \cap E(H_j)\big) = M_2^+ \cap \big(S \cap E(H_j)\big)$. Set
$$
M_1= \left(\bigcup^{n}_{\substack{t=1\\ t\ne j}} M_{H_t} \right)\cup M_G \cup M_1^+ 
\quad {\rm and} \quad 
M_2 = \left(\bigcup^{n}_{\substack{t=1\\ t\ne j}} M_{H_t} \right)\cup M_G \cup M_2^+.
$$
Since $H$ is a randomly matchable graph, each of $M_1^+$, $M_2^+$, and $M_{H_t}$, $t\ne j$, is a perfect matching (of $H$), hence no edge from $E_{H,G}$ can be added to $M_1$ or to $M_2$ to extend the matchings. Hence $M_1$ and $M_2$ are different maximal matchings of $G\circ H$. Now we have
\begin{align*}
M_1\cap S&= \left(\left(\bigcup^{n}_{\substack{t=1\\ t\ne j}} M_{H_t} \right)\cap S\right)\cup\left(M_G\cap S\right) \cup 						\left(M_1^+\cap S\right)\\
               &=\left(\left(\bigcup^{n}_{\substack{t=1\\ t\ne j}} M_{H_t} \right)\cap S\right)\cup\left(M_G\cap S\right) \cup 						\left(M_1^+\cap (S\cap E(H_j)\right)\\
               &=\left(\left(\bigcup^{n}_{\substack{t=1\\ t\ne j}} M_{H_t} \right)\cap S\right)\cup\left(M_G\cap S\right) \cup 						\left(M_2^+\cap (S\cap E(H_j)\right)\\
               &=\left(\left(\bigcup^{n}_{\substack{t=1\\ t\ne j}} M_{H_t} \right)\cap S\right)\cup\left(M_G\cap S\right) \cup 						\left(M_2^+\cap S\right)\\
               &=M_2 \cap S,
\end{align*}
a contradiction with the assumption that $S$ is a global forcing set of $G\circ H$. 

We have thus seen that $|S \cap E(H_i)| \ge \phi_{gm}(H)$ for each $i\in [n(G)]$. By a similar argument we also get that $S \cap E(G)| \ge \phi_{gm}(G)$. To prove the claimed bound it thus remains to prove that $|S\cap E_{G,H}|\geq n(G)n(H)/2$.

Assume, to the contrary, that $|S\cap E_{G,H}| < n(G)n(H)/2$. Then, by the Pigeonhole principle, there exist an index $i\in [n(G)]$ and two edges $g_ih^i_l$ and $g_ih^i_t$ such that $g_ih^i_l, g_ih^i_t\notin S$ and $h^i_lh^i_t\in M_{H_i}$. Let  $M_{G-g_i}$ be a maximal matching of $G-g_i$ and set 
\begin{align*}
M'&=\left(\bigcup^{n}_{\substack{j=1\\ j\ne i}} M_{H_j} \right)\cup M_{G-g_i} \cup (M_{H_i}-\{h^i_lh^i_t\})\cup\{g_ih^i_l\}\,,\\
M''&=\left(\bigcup^{n}_{\substack{j=1\\ j\ne l}} M^t \right)\cup M_{G-g_i} \cup (M_{H_i}-\{h^i_lh^i_t\})\cup\{g_ih^i_t\}\,.
\end{align*}
Then each of $M'$ and $M''$ is a maximal matching of $G\circ H$. Since $M'\cap S= M''\cap S$ we have a contradiction proving that $|S\cap E_{G,H}|\geq n(G)n(H)/2$. We conclude that 
\begin{align*}
\phi_{gm}(G\circ H)=|S|&=|S\cap E_H|+|S\cap E_G|+|S\cap E_{G,H}|\\
	 &\ge |V(G)|\phi_{gm}(H)+\phi_{gm}(G)+|S\cap E_{G,H}|\\
	 &\ge n(G)\phi_{gm}(H) + \phi_{gm}(G) + \frac{n(G)n(H)}{2}.
\end{align*}
\qed
\end{proof}

Sumner~\cite{sumner-1979} proved that even complete graphs $K_{2n}$ and balanced complete bipartite graphs $K_{n,n}$ are the only randomly matchable graphs. Since Vuki\v{c}evi\'c et al.\  proved that $\phi_{gm}(K_{2k}) = \frac{(2k-2)^2}{2}$ (\cite[Theorem 17]{global}) and that $\phi_{gm}(K_{k,k}) = (k-1)^2$ (\cite[Theorem 20]{global}), we have the following: 

\begin{corollary}
If $G$ is a graph and $k\ge 2$, then 
\begin{align*}
\phi_{gm}(G\circ K_k) & \ge \phi_{gm}(G) + n(G)(2k^2 - 3k + 2)\,,\\ 
\phi_{gm}(G\circ K_{k,k} & \ge \phi_{gm}(G) + n(G)(k^2 - k + 1)\,.
\end{align*}
\end{corollary}

%%%%%%%%%%%%%%%%%%%%%%%%%
\section{Integer programming}
\label{sec:integer}
%%%%%%%%%%%%%%%%%%%%%%%%%

In this short concluding section we present an integer linear programming model for finding a minimum global forcing set and the global forcing number. Its computational applicability is limited because the number of constraints is huge, nevertheless the model could be of theoretical interest. (In~\cite{Klavzar} a similar model for finding the edge metric dimension is presented.) 

Let $G$ be a graph with $E(G) = \{e_1,\ldots,e_{m(G)}\}$ and let $\{M_1, \ldots, M_t\}$ be the set of all maximal matchings of $G$. Let  $D_G=[d_{ij}]$ be the maximal matchings versus edges incidence matrix, that is, $D_G$ is a $t\times m$ matrix, where $d_{ij}= 1$ if $e_j\in M_i$, and $d_{ij}= 0$ otherwise. Let $F: \{0,1\}^{m(G)} \rightarrow {\mathbb N}_0$ be defined by 
\[F(x_1,\ldots, x_{m(G)}) = x_1 + \cdots + x_{m(G)}\,.\]
Then our goal is to determine $\min F$ subject to the constraints
\[|d_{i1}-d_{j1}|x_1+|d_{i2}-d_{j2}|x_2 + \cdots + |d_{im}-d_{jm}|x_m>0,\ 1 \le i < j \le t\,.\]
Note that if $x'_1, \ldots , x'_m$ is a set of values for which $F$ attains its minimum, then $W = \{e_i:\ x'_i =1\}$ is a minimum global forcing set of $G$. 

For example, consider the complete graph $K_3$ with the edge set $\{e_1,e_2,e_3\}$. Then $\{\{e_1\}, \{e_2\}, \{e_3\}\}$ is the set of all maximal matchings of $K_3$. The incidence matrix $D_{K_3}$ is then the identity matrix $I_3$. Thus, $\min F(x_1,x_2,x_3) = x_1+x_2+x_3$ subject to the  constraints $x_1+x_2 > 0$, $x_1 + x_3 > 0$, $x_2+x_3 > 0$, $x_1, x_2, x_3 \in \{0, 1\}$, has a solution $x_1 =  x_2 = 1$ and $x_3=0$. Hence $\{e_1,e_2\}$ is a minimum global forcing set for $K_3$.

\end{document}